\documentclass[12pt,a4paper]{article}%
\usepackage{amsfonts}
\usepackage{amssymb}
\usepackage{amsmath}
\usepackage{graphicx}%
\newtheorem{theorem}{Theorem}

\newtheorem{definition}[theorem]{Definition}

\newtheorem{lemma}[theorem]{Lemma}

\newenvironment{proof}[1][Proof]{\noindent\textbf{#1.} }{\ \rule{0.5em}{0.5em}}
\begin{document}

\title{Stochastic maximum principle for optimal control problem of backward systems
with terminal condition in $L^{1}\thanks{This work is partially supported by
Algerian-French cooperation, Tassili 07 MDU 705.}$}
\author{Seid Bahlali\\{\small Laboratory of Applied Mathematics,}\\{\small University Med Khider, Po. Box 145}\\{\small Biskra 07000, Algeria.}\\{\small sbahlali@yahoo.fr}}
\maketitle

\begin{abstract}
We consider a stochastic control problem, where the control domain is convex
and the system is governed by a nonlinear backward stochastic differential
equation.\ With a $L^{1}$\textbf{\ }terminal data, we derive necessary
optimality conditions in the form of stochastic maximum principle.

\ 


\ 


\end{abstract}

\section{Introduction}

We consider a stochastic control problem where the control domain is convex
and the system is governed by a backward stochastic differential equation
(BSDE\ for\ short) of the type%
\[
\left\{
\begin{array}
[c]{l}%
dy_{t}^{v}=b\left(  t,y_{t}^{v},z_{t}^{v},v_{t}\right)  dt+z_{t}^{v}dW_{t},\\
y_{T}^{v}=\xi,
\end{array}
\right.
\]
where $W=\left(  W_{t}\right)  _{t\geq0}$ is a standard Brownian motion,
defined on a filtered probability space $\left(  \Omega,\mathcal{F},\left(
\mathcal{F}_{t}\right)  _{t\geq0},\mathcal{P}\right)  ,$ satisfying the usual
conditions. The control variable $v$ is an $\mathcal{F}_{t}$-adapted process
with values in a convex closed subset $U$ of $\mathbb{R}^{m}$. The terminal
condition $\xi$ is a $n$-dimensional $\mathcal{F}_{T}$-measurable random
vector such that $\mathbb{E}\left\vert \xi\right\vert <\infty.$

The objective of the control problem, is to choose $u$ in such a way as to
minimize a functional cost of the type%
\[
J\left(  v\right)  =\mathbb{E}\left[  g\left(  y_{0}^{v}\right)  +%
{\displaystyle\int\nolimits_{0}^{T}}
h\left(  t,y_{t}^{v},z_{t}^{v},v_{t}\right)  dt\right]  .
\]

A control process that solves this problem is called optimal.

Stochastic control problems for the backward and forward-backward systems have
been studied by many authors including Peng $\left[  21\right]  $, Xu $\left[
24\right]  $, El-Karoui et al $\left[  12\right]  $,\ Wu $\left[  23\right]
$, Dokuchaev and Zhou $\left[  9\right]  $, Peng and Wu $\left[  22\right]  ,$
Bahlali and Labed $\left[  1\right]  $, Bahlali $\left[  2,3\right]  $.
Approachs based on dynamic programming have been studied by Fuhrman and
Tessetore $\left[  14\right]  $. All this papers consider BSDEs with $L^{p}$
terminal condition, $p\geq2$.

The aim of the present paper is to derive necessary optimality conditions, in
the form of stochastic maximum principle. The terminal condition is assumed in
$L^{1}$. This is the first version which covers the control of backward
systems in $L^{1}$. Our result extend all the previous works in the subject.

Since the control domain is convex, a classical way of treating such a problem
consists to use the convex perturbation method. More precisely, if $u $ is an
optimal control and $v$ is arbitrary, we define, for each $t\in\left[
0,T\right]  $, a perturbed control as follows%
\[
u^{\theta}=u+\theta\left(  v-u\right)  .
\]

With a sufficiently small $\theta>0$, we derive the variational equation from
the fact that
\[
0\leq J\left(  u^{\theta}\right)  -J\left(  u\right)  .
\]

The paper is organized as follows. In Section 2, we formulate the problem and
give the various assumptions used throughout the paper. Section 3 is devoted
to some preliminary results, which will be used in the sequel. In the last
Section, we derive our main result, the necessary optimality conditions.

\ 

Along this paper, we denote by $C$ some positive constant and for simplicity,
we need the following matrix notation. We denote by $\mathcal{M}_{n\times
d}\left(  \mathbb{R}\right)  $ the space of $n\times d$ real matrix and
$\mathcal{M}_{n\times n}^{d}\left(  \mathbb{R}\right)  $ the linear space of
vectors $M=\left(  M_{1},...,M_{d}\right)  $ where $M_{i}\in\mathcal{M}%
_{n\times n}\left(  \mathbb{R}\right)  $.

For any $M,N\in\mathcal{M}_{n\times n}^{d}\left(  \mathbb{R}\right)  $,
$L,S\in\mathcal{M}_{n\times d}\left(  \mathbb{R}\right)  $, $\alpha,\beta
\in\mathbb{R}^{n}$ and $\gamma\in\mathbb{R}^{d},$ we use the following notations

$\alpha\beta=%
{\displaystyle\sum\limits_{i=1}^{n}}
\alpha_{i}\beta_{i}\in\mathbb{R}$ is the product scalar in $\mathbb{R}^{n}$,

$LS=%
{\displaystyle\sum\limits_{i=1}^{d}}
L_{i}S_{i}\in\mathbb{R}$, where $L_{i}$ and\ $S_{i}$ are the $i^{th}$ columns
of $L$ and $S,$

$ML=%
{\displaystyle\sum\limits_{i=1}^{d}}
M_{i}L_{i}\in\mathbb{R}^{n}$,

$M\alpha\gamma=\sum\limits_{i=1}^{d}\left(  M_{i}\alpha\right)  \gamma_{i}%
\in\mathbb{R}^{n}$,

$MN=%
{\displaystyle\sum\limits_{i=1}^{d}}
M_{i}N_{i}\in\mathcal{M}_{n\times n}\left(  \mathbb{R}\right)  $,

$MLN=%
{\displaystyle\sum\limits_{i=1}^{d}}
M_{i}LN_{i}\in\mathcal{M}_{n\times n}\left(  \mathbb{R}\right)  $,

$ML\gamma=%
{\displaystyle\sum\limits_{i=1}^{d}}
M_{i}L\gamma_{i}\in\mathcal{M}_{n\times n}\left(  \mathbb{R}\right)  $.

We denote by $L^{\ast}$ the transpose of the matrix $L$ and $M^{\ast}=\left(
M_{1}^{\ast},...,M_{d}^{\ast}\right)  $.

\section{Formulation of the problem}

Let $T$ be a fixed strictly positive real number and $\left(  \Omega
,\mathcal{F},\left(  \mathcal{F}_{t}\right)  _{t\in\left[  0,T\right]
},\mathcal{P}\right)  $ be a filtered probability space satisfying the usual
conditions, on which a $d$-dimensional Brownian motion $W=\left(
W_{t}\right)  _{t\in\left[  0,T\right]  }$\ is defined. We assume that
$\left(  \mathcal{F}_{t}\right)  _{t\in\left[  0,T\right]  }$ is the
$\mathcal{P}$- augmentation of the natural filtration of $\left(
W_{t}\right)  _{t\in\left[  0,T\right]  }$.

\begin{definition}
Let $U$ be a closed convex subset of $\mathbb{R}^{m}$. \textit{An admissible
control }$v$ \textit{is an }$\mathcal{F}_{t}$-\textit{adapted process with
values in }$U$\textit{\ such that }%
\[
\underset{t\in\left[  0,T\right]  }{\sup}\mathbb{E}\left\vert v_{t}\right\vert
^{2}<\infty.
\]

\textit{We denote by }$\mathcal{U}$\textit{\ the set of all admissible
controls.}
\end{definition}

For any $v\in\mathcal{U}$, we consider the following controlled BSDE
\begin{equation}
\left\{
\begin{array}
[c]{l}%
dy_{t}^{v}=b\left(  t,y_{t}^{v},z_{t}^{v},v_{t}\right)  dt+z_{t}^{v}dW_{t},\\
y_{T}^{v}=\xi,
\end{array}
\right.
\end{equation}
where $b:\left[  0,T\right]  \times\mathbb{R}^{n}\times\mathcal{M}_{n\times
d}\left(  \mathbb{R}\right)  \times U\longrightarrow\mathbb{R}^{n}$ and $\xi$
is an $n$-dimensional $\mathcal{F}_{T}$-\ measurable random vector such that
$\mathbb{E}\left\vert \xi\right\vert <\infty.$

The aim of the control problem is to minimize, over the class $\mathcal{U}$ of
admissible controls, a functional cost of the form
\begin{equation}
J\left(  v\right)  =\mathbb{E}\left[  g\left(  y_{0}^{v}\right)  +%
{\displaystyle\int\nolimits_{0}^{T}}
h\left(  t,y_{t}^{v},z_{t}^{v},v_{t}\right)  dt\right]  ,
\end{equation}
where $g:\mathbb{R}^{n}\longrightarrow\mathbb{R}$ and $h:\left[  0,T\right]
\times\mathbb{R}^{n}\times\mathcal{M}_{n\times d}\left(  \mathbb{R}\right)
\times U\longrightarrow\mathbb{R}$.

\ 

A control $u\in\mathcal{U}$ is called optimal, if that solves the problem
\begin{equation}
J(u)=\inf\limits_{v\in\mathcal{U}}J(v).
\end{equation}

Our goal in this paper is to establish necessary optimality conditions, in the
form of stochastic maximum principle.

\ 

To study this kind of problem, we need reasonable conditions which ensure the
existence and uniqueness of solutions of BSDEs with $L^{1}$ terminal
condition. This is given by the results of Briand et al $\left[
5,\ \text{page 124-128}\right]  $.

Miming $\left[  5\right]  $, we use the following notations.

Let us denote by $\sum_{T}$ the set of all stopping times $\tau$ such that
$\tau\leq T$. A process $Y=\left(  Y_{t}\right)  _{t\in\left[  0,T\right]  }$
belongs to class $\left(  D\right)  ,$ if the family $\left\{  Y_{\tau}%
,\ \tau\in\sum_{T}\right\}  $ is uniformly integrable.

For a process $Y$ in class $\left(  D\right)  $, we put%
\[
\left\Vert Y\right\Vert _{1}=\sup\left\{  \mathbb{E}\left\vert Y_{\tau
}\right\vert ,\ \tau\in\sum_{T}\right\}  .
\]

The space of progressively measurable continuous processes which belong to
class $\left(  D\right)  $ is complete under this norm, see Dellacherie and
Meyer $\left[  7,\ \text{page 90}\right]  .$

For any real $p>0$, $S^{p}=S^{p}\left(  \mathbb{R}^{n}\right)  $ denotes the
set of $\mathbb{R}^{n}$-valued, adapted cadlag processes $\left\{
X_{t}\right\}  _{t\in\left[  0,T\right]  }$ such that
\[
\left\Vert X\right\Vert _{S^{p}}=\mathbb{E}\left[  \underset{t}{\sup
}\left\vert X_{t}\right\vert ^{p}\right]  ^{1\wedge1/p}<+\infty.
\]

If $p\geq1$, $\left\Vert .\right\Vert _{S^{p}}$ is a norm on $S^{p}$ and if
$p\in\left(  0,1\right)  $, $\left(  X,X^{^{\prime}}\right)  \longmapsto
\left\Vert X-X^{^{\prime}}\right\Vert _{S^{p}}$ defines a distance on $S^{p}
$. Under this metric, $S^{p}$ is complete.

$M^{p}=M^{p}\left(  \mathbb{R}^{n}\right)  $ denotes the set of (equivalent
classes of) predictable processes $\left\{  X_{t}\right\}  _{t\in\left[
0,T\right]  }$ with values in $\mathbb{R}^{n}$ such that
\[
\left\Vert X\right\Vert _{M^{p}}=\mathbb{E}\left[  \left(
{\displaystyle\int\nolimits_{0}^{T}}
\left\vert X_{t}\right\vert ^{2}dt\right)  ^{p/2}\right]  ^{1\wedge
1/p}<+\infty.
\]

For $p\geq1$, $M^{p}$ is a Banach space endowed with this norm and for
$p\in\left(  0,1\right)  $, $M^{p}$ is a complete metric space with the
resulting distance.

\ 

We assume,%
\begin{equation}%
\begin{array}
[c]{l}%
\text{(4.1) }b,g,h\ \text{are continuously differentiable with respect to
}\left(  y,z,v\right)  \text{.}\\
\text{(4.2) The derivatives }b_{y},b_{z},b_{v},h_{y},h_{z},h_{v}\text{ and
}g_{y}\text{ are continuous }\\
\ \ \ \ \ \text{in }\left(  y,z,v\right)  \text{ and uniformly bounded.}\\
\text{(4.3) }g\text{\ is bounded by }C\left(  1+\left\vert y\right\vert
\right)  .\\
\text{(4.4) }\forall r>0,\text{ we have (for }f=b,h\text{)}\\
\ \ \ \ \ \phi_{r}\left(  t\right)  :=\underset{\left\vert y\right\vert \leq
r}{\sup}\left\vert f\left(  t,y,0,v\right)  -f\left(  t,0,0,v\right)
\right\vert \in L^{1}\left(  \left[  0,T\right]  \times\Omega,m\otimes
\mathcal{P}\right)  .\\
\text{(4.5) There exists two constants }C\geq0,\ \alpha\in\left(  0,1\right)
\text{ and a non-negative }\\
\ \ \ \ \ \text{progressively measurable processes }\left\{  \varphi
_{t}\right\}  _{t\in\left[  0,T\right]  }\text{ and }\left\{  \psi
_{t}\right\}  _{t\in\left[  0,T\right]  }\text{ }\\
\ \ \ \ \ \text{such that }\forall\left(  t,y,z,v\right)  \in\left[
0,T\right]  \times\mathbb{R}^{n}\times\mathcal{M}_{n\times d}\left(
\mathbb{R}\right)  \times U,\\
\ \ \ \ \ \left\vert f\left(  t,y,z,v\right)  -f\left(  t,y,0,v\right)
\right\vert \leq C\left(  \varphi_{t}+\left\vert y\right\vert +\left\vert
z\right\vert +\left\vert v\right\vert \right)  ^{\alpha},\text{ for }f=b,h.\\
\ \ \ \ \ \mathbb{E}\left[  \left\vert \xi\right\vert +%
{\displaystyle\int\nolimits_{0}^{T}}
\left(  \varphi_{t}+\psi_{t}\right)  dt\right]  <+\infty.\\
\text{(4.6) }\ \forall\left(  t,y,z_{1},v\right)  ,\left(  t,y,z_{2},v\right)
\in\left[  0,T\right]  \times\mathbb{R}^{n}\times\mathcal{M}_{n\times
d}\left(  \mathbb{R}\right)  \times U,\\
\ \ \ \ \ \left\vert f\left(  t,y,z_{1},v\right)  -f\left(  t,y,z_{2}%
,v\right)  \right\vert \leq C\left\vert z_{1}-z_{2}\right\vert ,\ \\
\ \ \ \ \ \text{for }f=b_{y},b_{z},b_{v},h_{y},h_{z},h_{v}.
\end{array}
\end{equation}

The above assumptions imply those of Briand et al $\left[  5\right]  $. Hence
from $\left[  5\text{, Th 6.2, p 125 and Th 6.3, p 126}\right]  $, for every
$v\in\mathcal{U}$, equation $\left(  1\right)  $\ admits a unique adapted solution.

We note that for the uniqueness, the solution $y$ belongs to the class
$\left(  D\right)  $ and $z$ belongs to the space $%
{\displaystyle\bigcup\nolimits_{\beta>\alpha}}
M^{\beta}$, $\alpha\in\left(  0,1\right)  $. For the existence, the solution
$y$ belongs to the class $\left(  D\right)  $ and for each $\beta\in\left(
0,1\right)  $, $\left(  y,z\right)  $ belongs to the space $S^{\beta}\times
M^{\beta}$.

More details are given in Briand et al $\left[  5,\text{ page 124-128}\right]
$.

\ 

To enclose the formulation of the problem, it remains us to prove that the
cost $J$ is well defined. This is given by the following lemma.

\begin{lemma}
The functional cost $J$ is well defined from $\mathcal{U}$ into $\mathbb{R}$.
\end{lemma}

\begin{proof}
Consider the following controlled one dimensional BSDE%
\[
\left\{
\begin{array}
[c]{l}%
dx_{t}^{v}=h\left(  t,y_{t}^{v},z_{t}^{v},v_{t}\right)  dt+k_{t}^{v}dW_{t},\\
x_{T}^{v}=\eta.
\end{array}
\right.
\]
where $k^{v}=\left(  k_{1}^{v},...,k_{d}^{v}\right)  $ is an $\left(  1\times
d\right)  $ real matrix, $\left(  y^{v},z^{v}\right)  $ is the solution of
equation $\left(  1\right)  $ and $\eta$ is a one dimensional $\mathcal{F}_{T}
$-measurable random variable such that $\mathbb{E}\left\vert \eta\right\vert
<\infty.$

Under assumptions $\left(  4\right)  $, the above one dimensional BSDE\ admits
a unique adapted solution $\left(  x^{v},k^{v}\right)  $.

We put%
\[
\widetilde{y}=\left(
\begin{array}
[c]{c}%
y^{v}\\
x^{v}%
\end{array}
\right)  ,
\]
and consider now the following $\left(  n+1\right)  $-dimensional BSDE%
\[
\left\{
\begin{array}
[c]{l}%
d\widetilde{y}_{t}=\widetilde{b}\left(  t,\widetilde{y}_{t},\widetilde{z}%
_{t},v_{t}\right)  dt+\widetilde{z}_{t}dW_{t},\\
\widetilde{y}_{T}=\left(
\begin{array}
[c]{c}%
\xi\\
\eta
\end{array}
\right)  ,
\end{array}
\right.
\]
where the function $\widetilde{b}$ is defined from $\left[  0,T\right]
\times\mathbb{R}^{n+1}\times\mathcal{M}_{\left(  n+1\right)  \times d}\left(
\mathbb{R}\right)  \times U$ into $\mathbb{R}^{n+1}$ by%
\[
\widetilde{b}\left(  t,\widetilde{y}_{t},\widetilde{z}_{t},v_{t}\right)
=\left(
\begin{array}
[c]{c}%
b\left(  t,y_{t}^{v},z_{t}^{v},v_{t}\right) \\
h\left(  t,y_{t}^{v},z_{t}^{v},v_{t}\right)
\end{array}
\right)  ,
\]
and $\widetilde{z}$ is a $\left(  n+1\right)  \times d$ real matrix given by%
\[
\widetilde{z}=\left(
\begin{array}
[c]{c}%
z^{v}\\
k^{v}%
\end{array}
\right)  =\left(
\begin{array}
[c]{c}%
z_{11}^{v}\ \ z_{12}^{v}\ ...\ z_{1d}^{v}\\
z_{21}^{v}\ \ z_{22}^{v}\ ...\ z_{2d}^{v}\\
\vdots\ \ \ \ \ \ \ \ \ \ \ \ \ \ \ \ \vdots\\
z_{n1}^{v}\ \ z_{n2}^{v}\ ...\ z_{nd}^{v}\\
k_{1}^{v}\ \ k_{2}^{v}\ ...\ k_{d}^{v}%
\end{array}
\right)  ,
\]

It's obvious that $\widetilde{b}$\ satisfies hypothesis $\left(  4\right)  $,
then the above $\left(  n+1\right)  $-dimensional BSDE admits a unique adapted
solution $\left(  \widetilde{y}_{t},\widetilde{z}_{t}\right)  $.

Define now the function $\widetilde{g}$ from $\mathbb{R}^{n+1}$ into
$\mathbb{R}$ by%
\[
\widetilde{g}\left(  \widetilde{y}_{t}\right)  =g\left(  y_{t}^{v}\right)
-x_{t}^{v},
\]
and the new functional cost from $\mathcal{U}$ into $\mathbb{R}$ by
\[
\widetilde{J}\left(  v\right)  =\mathbb{E}\left[  \widetilde{g}\left(
\widetilde{y}_{0}\right)  \right]  +\mathbb{E}\left[  \eta\right]  .
\]

It's easy to see that for every $v\in\mathcal{U}$
\[
\widetilde{J}\left(  v\right)  =J\left(  v\right)  .
\]

By $\left(  4.3\right)  $ the cost $\widetilde{J}$ is well defined from
$\mathcal{U}$ into $\mathbb{R}$ and since $\widetilde{J}\left(  v\right)
=J\left(  v\right)  $, for every $v\in\mathcal{U}$, the cost $J$ is well
defined from $\mathcal{U}$ into $\mathbb{R}$.

The proof is completed.
\end{proof}

\ 

Let us now state and prove an alternative result that we will be used along
this paper. This result said that the difference betwen two solutions of BSDEs
with the same terminal condition in $L^{1}$ is a solution of BSDE in $L^{2}$,
and it is given by the following lemma.

\begin{lemma}
Let $\left(  y^{v},z^{v}\right)  $\textit{\ and }$\left(  y^{w},z^{w}\right)
$\textit{\ be the solutions of }$\left(  1\right)  $\textit{\ associated
respectively with the controls }$v$ and $w.$\textit{\ Then the following BSDE}%
\[
\left\{
\begin{array}
[c]{l}%
d\left(  y_{t}^{v}-y_{t}^{w}\right)  =\left[  b\left(  t,y_{t}^{v},z_{t}%
^{v},v_{t}\right)  -b\left(  t,y_{t}^{w},z_{t}^{w},w_{t}\right)  \right]
dt+\left(  z_{t}^{v}-z_{t}^{w}\right)  dW_{t},\\
y_{T}^{v}-y_{T}^{w}=0,
\end{array}
\right.
\]
\textit{admits a unique adapted solution }$\left(  y^{v}-y^{w},z^{v}%
-z^{w}\right)  $\textit{\ such that}%
\begin{equation}
\underset{t\in\left[  0,T\right]  }{\sup}\mathbb{E}\left\vert y_{t}^{v}%
-y_{t}^{w}\right\vert ^{2}+\mathbb{E}%
{\displaystyle\int\nolimits_{0}^{T}}
\left\vert z_{t}^{v}-z_{t}^{w}\right\vert ^{2}dt<+\infty.
\end{equation}

\end{lemma}

\begin{proof}
We have%
\[
y_{t}^{v}-y_{t}^{w}=-%
{\displaystyle\int\nolimits_{t}^{T}}
\left[  b\left(  s,y_{s}^{v},z_{s}^{v},v_{s}\right)  -b\left(  s,y_{s}%
^{w},z_{s}^{w},w_{s}\right)  \right]  ds-%
{\displaystyle\int\nolimits_{t}^{T}}
\left(  z_{s}^{v}-z_{s}^{w}\right)  dW_{s}.
\]

Then%
\begin{align*}
&  y_{t}^{v}-y_{t}^{w}\\
&  =-%
{\displaystyle\int\nolimits_{t}^{T}}
\left(
{\displaystyle\int\nolimits_{0}^{1}}
b_{y}\left(  s,y_{s}^{w}+\lambda\left(  y_{s}^{v}-y_{s}^{w}\right)  ,z_{s}%
^{w}+\lambda\left(  z_{s}^{v}-z_{s}^{w}\right)  ,w_{s}+\lambda\left(
v_{s}-w_{s}\right)  \right)  d\lambda\right)  \left(  y_{s}^{v}-y_{s}%
^{w}\right)  ds\\
&  -\left(
{\displaystyle\int\nolimits_{0}^{1}}
b_{z}\left(  s,y_{s}^{w}+\lambda\left(  y_{s}^{v}-y_{s}^{w}\right)  ,z_{s}%
^{w}+\lambda\left(  z_{s}^{v}-z_{s}^{w}\right)  ,w_{s}+\lambda\left(
v_{s}-w_{s}\right)  \right)  d\lambda\right)  \left(  z_{s}^{v}-z_{s}%
^{w}\right)  ds\\
&  -%
{\displaystyle\int\nolimits_{t}^{T}}
\left(
{\displaystyle\int\nolimits_{0}^{1}}
b_{v}\left(  s,y_{s}^{w}+\lambda\left(  y_{s}^{v}-y_{s}^{w}\right)  ,z_{s}%
^{w}+\lambda\left(  z_{s}^{v}-z_{s}^{w}\right)  ,w_{s}+\lambda\left(
v_{s}-w_{s}\right)  \right)  d\lambda\right)  \left(  v_{s}-w_{s}\right)  ds\\
&  -%
{\displaystyle\int\nolimits_{t}^{T}}
\left(  z_{s}^{v}-z_{s}^{w}\right)  dW_{s}.
\end{align*}

The above equation is a linear BSDE.\ Since $b_{y},b_{z},b_{v}$ are bounded,
the terminal condition $y_{T}^{v}-y_{T}^{w}=0$ and the controls are in $L^{2}
$, then by a classical result on BSDEs (see Pardoux-Peng $\left[  19\right]
$, El Karoui et al $\left[  12\right]  $), we have the desired results.
\end{proof}

\section{Preliminary results}

Since the control domain $U$ is convex, the classical way consists to use the
convex perturbation method. More precisely, let $u$ be an optimal control
minimizing the cost $J$ over $\mathcal{U}$ and $\left(  y_{t}^{u},z_{t}%
^{u}\right)  $ the solution of $\left(  1\right)  $ controlled by $u$. Define
a perturbed control as follows
\[
u_{t}^{\theta}=u_{t}+\theta\left(  v_{t}-u_{t}\right)  ,
\]
where $\theta>0$ is sufficiently small and $v$ is an arbitrary element of
$\mathcal{U}$.

It's clear that $u^{\theta}$ is an element of $\mathcal{U}$ (admissible control).

Denote by $\left(  y_{t}^{\theta},z_{t}^{\theta}\right)  $ the solution of
$\left(  1\right)  $ associated with $u^{\theta}$.

\ 

Since $u$ is optimal, the variational inequality follows from the fact that%
\[
0\leq J\left(  u^{\theta}\right)  -J\left(  u\right)  .
\]

This is can be proved by using the following lemmas.

\begin{lemma}
\textit{Under assumptions }$\left(  4\right)  $, we have%
\begin{equation}
\underset{\theta\rightarrow0}{\lim}\left(  \underset{t\in\left[  0,T\right]
}{\sup}\mathbb{E}\left\vert y_{t}^{\theta}-y_{t}^{u}\right\vert ^{2}%
+\mathbb{E}%
{\displaystyle\int\nolimits_{0}^{T}}
\left\vert z_{t}^{\theta}-z_{t}^{u}\right\vert ^{2}dt\right)  =0.
\end{equation}

\end{lemma}

\begin{proof}
By $\left(  5\right)  $, we have%
\[
\underset{t\in\left[  0,T\right]  }{\sup}\mathbb{E}\left\vert y_{t}^{\theta
}-y_{t}^{u}\right\vert ^{2}+\mathbb{E}%
{\displaystyle\int\nolimits_{0}^{T}}
\left\vert z_{t}^{\theta}-z_{t}^{u}\right\vert ^{2}dt<+\infty.
\]

Applying the Ito formula to $\left(  y_{t}^{\theta}-y_{t}^{u}\right)  ^{2}$,
we get%
\begin{align*}
&  \mathbb{E}\left\vert y_{t}^{\theta}-y_{t}^{u}\right\vert ^{2}+\mathbb{E}%
{\displaystyle\int\nolimits_{t}^{T}}
\left\vert z_{s}^{\theta}-z_{s}^{u}\right\vert ^{2}ds\\
&  =2\mathbb{E}%
{\displaystyle\int\nolimits_{t}^{T}}
\left\vert \left(  y_{s}^{\theta}-y_{s}^{u}\right)  \left(  b\left(
s,y_{s}^{\theta},z_{s}^{\theta},u_{s}^{\theta}\right)  -b\left(  s,y_{s}%
^{u},z_{s}^{u},u_{s}\right)  \right)  \right\vert \,ds\\
&  \leq2\mathbb{E}%
{\displaystyle\int\nolimits_{t}^{T}}
\left\vert \left(  y_{s}^{\theta}-y_{s}^{u}\right)  \left(  b\left(
s,y_{s}^{\theta},z_{s}^{\theta},u_{s}^{\theta}\right)  -b\left(  s,y_{s}%
^{u},z_{s}^{u},u_{s}^{\theta}\right)  \right)  \right\vert \,ds\\
&  +2\mathbb{E}%
{\displaystyle\int\nolimits_{t}^{T}}
\left\vert \left(  y_{s}^{\theta}-y_{s}^{u}\right)  \left(  b\left(
s,y_{s}^{u},z_{s}^{u},u_{s}^{\theta}\right)  -b\left(  s,y_{s}^{u},z_{s}%
^{u},u_{s}\right)  \right)  \right\vert \,ds.
\end{align*}

Applying the Young's formula to the first term in the right hand side of the
above inequality, we have for every $\varepsilon>0$%
\begin{align*}
&  \mathbb{E}\left\vert y_{t}^{\theta}-y_{t}^{u}\right\vert ^{2}+\mathbb{E}%
{\displaystyle\int\nolimits_{t}^{T}}
\left\vert z_{s}^{\theta}-z_{s}^{u}\right\vert ^{2}ds\\
&  \leq%
\genfrac{.}{.}{}{0}{1}{\varepsilon}%
\mathbb{E}%
{\displaystyle\int\nolimits_{t}^{T}}
\left\vert y_{s}^{\theta}-y_{s}^{u}\right\vert ^{2}ds+\varepsilon\mathbb{E}%
{\displaystyle\int\nolimits_{t}^{T}}
\left\vert b\left(  s,y_{s}^{\theta},z_{s}^{\theta},u_{s}^{\theta}\right)
-b\left(  s,y_{s}^{u},z_{s}^{u},u_{s}^{\theta}\right)  \right\vert ^{2}ds\\
&  +2\mathbb{E}%
{\displaystyle\int\nolimits_{t}^{T}}
\left\vert \left(  y_{s}^{\theta}-y_{s}^{u}\right)  \left(  b\left(
s,y_{s}^{u},z_{s}^{u},u_{s}^{\theta}\right)  -b\left(  s,y_{s}^{u},z_{s}%
^{u},u_{s}\right)  \right)  \right\vert \,ds.
\end{align*}

By $\left(  4.2\right)  $, $b$ is uniformly Lipschitz with respect $\left(
y,z,v\right)  $. Then%
\begin{align*}
&  \mathbb{E}\left\vert y_{t}^{\theta}-y_{t}^{u}\right\vert ^{2}+\mathbb{E}%
{\displaystyle\int\nolimits_{t}^{T}}
\left\vert z_{s}^{\theta}-z_{s}^{u}\right\vert ^{2}ds\\
&  \leq\left(
\genfrac{.}{.}{}{0}{1}{\varepsilon}%
+C\varepsilon\right)
{\displaystyle\int\nolimits_{t}^{T}}
\mathbb{E}\left\vert y_{s}^{\theta}-y_{s}^{u}\right\vert ^{2}ds+C\varepsilon%
{\displaystyle\int\nolimits_{t}^{T}}
\mathbb{E}\left\vert z_{s}^{\theta}-z_{s}^{u}\right\vert ^{2}ds\\
&  +C\theta%
{\displaystyle\int\nolimits_{t}^{T}}
\mathbb{E}\left[  \left\vert y_{s}^{\theta}-y_{s}^{u}\right\vert \left\vert
v_{s}-u_{s}\right\vert \right]  ds.
\end{align*}

Applying the Cauchy-Schwarz inequality to the third term in the right hand
side of the above inequality, we get%
\begin{align*}
&  \mathbb{E}\left\vert y_{t}^{\theta}-y_{t}^{u}\right\vert ^{2}+\mathbb{E}%
{\displaystyle\int\nolimits_{t}^{T}}
\left\vert z_{s}^{\theta}-z_{s}^{u}\right\vert ^{2}ds\\
&  \leq\left(
\genfrac{.}{.}{}{0}{1}{\varepsilon}%
+C\varepsilon\right)
{\displaystyle\int\nolimits_{t}^{T}}
\mathbb{E}\left\vert y_{s}^{\theta}-y_{s}^{u}\right\vert ^{2}ds+C\varepsilon%
{\displaystyle\int\nolimits_{t}^{T}}
\mathbb{E}\left\vert z_{s}^{\theta}-z_{s}^{u}\right\vert ^{2}ds\\
&  +C\theta\left(
{\displaystyle\int\nolimits_{t}^{T}}
\mathbb{E}\left\vert y_{s}^{\theta}-y_{s}^{u}\right\vert ^{2}ds\right)
^{1/2}\left(
{\displaystyle\int\nolimits_{t}^{T}}
\mathbb{E}\left\vert v_{s}-u_{s}\right\vert 2ds\right)  ^{1/2}.
\end{align*}

Using definition 1 and $\left(  6\right)  $, we have%
\begin{align*}
&  \mathbb{E}\left\vert y_{t}^{\theta}-y_{t}^{u}\right\vert ^{2}+\mathbb{E}%
{\displaystyle\int\nolimits_{t}^{T}}
\left\vert z_{s}^{\theta}-z_{s}^{u}\right\vert ^{2}ds\\
&  \leq\left(
\genfrac{.}{.}{}{0}{1}{\varepsilon}%
+C\varepsilon\right)
{\displaystyle\int\nolimits_{t}^{T}}
\mathbb{E}\left\vert y_{s}^{\theta}-y_{s}^{u}\right\vert ^{2}ds+C\varepsilon%
{\displaystyle\int\nolimits_{t}^{T}}
\mathbb{E}\left\vert z_{s}^{\theta}-z_{s}^{u}\right\vert ^{2}ds\\
&  +C\varepsilon%
{\displaystyle\int\nolimits_{t}^{T}}
\mathbb{E}\left\vert z_{s}^{\theta}-z_{s}^{u}\right\vert ^{2}ds+C\theta.
\end{align*}

Choose $\varepsilon=%
\genfrac{.}{.}{}{0}{1}{2C}%
$, then we get%
\[
\mathbb{E}\left\vert y_{t}^{\theta}-y_{t}^{u}\right\vert ^{2}+%
\genfrac{.}{.}{}{0}{1}{2}%
\mathbb{E}%
{\displaystyle\int\nolimits_{t}^{T}}
\left\vert z_{s}^{\theta}-z_{s}^{u}\right\vert ^{2}ds\leq\left(  2C+%
\genfrac{.}{.}{}{0}{1}{2}%
\right)
{\displaystyle\int\nolimits_{t}^{T}}
\mathbb{E}\left\vert y_{s}^{\theta}-y_{s}^{u}\right\vert ^{2}ds+C\theta.
\]

From this above inequality, we deduce two inequalities%
\begin{equation}
\mathbb{E}\left\vert y_{t}^{\theta}-y_{t}^{u}\right\vert ^{2}\leq\left(  2C+%
\genfrac{.}{.}{}{0}{1}{2}%
\right)
{\displaystyle\int\nolimits_{t}^{T}}
\mathbb{E}\left\vert y_{s}^{\theta}-y_{s}^{u}\right\vert ^{2}ds+C\theta.
\end{equation}%
\begin{equation}
\mathbb{E}%
{\displaystyle\int\nolimits_{t}^{T}}
\left\vert z_{s}^{\theta}-z_{s}^{u}\right\vert ^{2}ds\leq\left(  4C+1\right)
{\displaystyle\int\nolimits_{t}^{T}}
\mathbb{E}\left\vert y_{s}^{\theta}-y_{s}^{u}\right\vert ^{2}ds+C\theta.
\end{equation}

By $\left(  7\right)  $, Gronwall lemma and Buckholers-Davis-Gundy inequality,
we have%
\[
\underset{\theta\rightarrow0}{\lim}\left(  \underset{t\in\left[  0,T\right]
}{\sup}\mathbb{E}\left\vert y_{t}^{\theta}-y_{t}^{u}\right\vert ^{2}\right)
=0.
\]

Finally, by $\left(  8\right)  $ and the above result, we obtain%
\[
\underset{\theta\rightarrow0}{\lim}\mathbb{E}%
{\displaystyle\int\nolimits_{t}^{T}}
\left\vert z_{s}^{\theta}-z_{s}^{u}\right\vert ^{2}ds=0.
\]

The lemma is proved.
\end{proof}

\begin{lemma}
For every $v\in\mathcal{U}$, the\textit{\ following linear BSDE}%
\begin{equation}
\left\{
\begin{array}
[c]{ll}%
dY_{t}= & \left[  b_{y}\left(  t,y_{t}^{u},z_{t}^{u},u_{t}\right)  Y_{t}%
+b_{z}\left(  t,y_{t}^{u},z_{t}^{u},u_{t}\right)  Z_{t}\right]  dt\\
& b_{v}\left(  t,y_{t}^{u},z_{t}^{u},u_{t}\right)  \left(  v_{t}-u_{t}\right)
dt+Z_{t}dW_{t},\\
Y_{T}= & 0,
\end{array}
\right.
\end{equation}
admits a unique\ adapted solution $\left(  Y,Z\right)  $\textit{\ such that}%
\begin{equation}
\underset{t\in\left[  0,T\right]  }{\sup}\mathbb{E}\left\vert Y_{t}\right\vert
^{2}+\mathbb{E}%
{\displaystyle\int\nolimits_{0}^{T}}
\left\vert Z_{t}\right\vert ^{2}dt<\infty.
\end{equation}%
\begin{equation}
\underset{\theta\rightarrow0}{\lim}\left(  \mathbb{E}\left\vert Y_{t}-%
\genfrac{.}{.}{}{0}{y_{t}^{\theta}-y_{t}^{u}}{\theta}%
\right\vert ^{2}+\mathbb{E}%
{\displaystyle\int\nolimits_{0}^{T}}
\left\vert Z_{t}-%
\genfrac{.}{.}{}{0}{z_{t}^{\theta}-z_{t}^{u}}{\theta}%
\right\vert ^{2}dt\right)  =0.
\end{equation}

\end{lemma}

\begin{proof}
i) Assertion $\left(  10\right)  $ is obvious since the BSDE $\left(
9\right)  $ is linear, $b_{y},$ $b_{z},b_{v}$ are bounded and the terminal
condition $Y_{T}=0.$

ii) Let us prove $\left(  11\right)  $.

Put%
\begin{align*}
\Phi_{t}^{\theta}  &  =Y_{t}-%
\genfrac{.}{.}{}{0}{y_{t}^{\theta}-y_{t}^{u}}{\theta}%
,\\
\Psi_{t}^{\theta}  &  =Z_{t}-%
\genfrac{.}{.}{}{0}{z_{t}^{\theta}-z_{t}^{u}}{\theta}%
.
\end{align*}

We have%
\[
d\Phi_{t}^{\theta}=\left[  B_{y}^{\theta}\left(  t\right)  \Phi_{t}^{\theta
}+B_{z}^{\theta}\left(  t\right)  \Psi_{t}^{\theta}+\rho_{t}^{\theta}\right]
dt+\Psi_{t}^{\theta}dW_{t},
\]
where
\begin{align*}
B_{y}^{\theta}\left(  t\right)   &  =%
{\displaystyle\int\nolimits_{0}^{1}}
b_{y}\left(  t,y_{t}^{u}+\lambda\left(  y_{t}^{\theta}-y_{t}^{u}\right)
,z_{t}^{u}+\lambda\left(  z_{t}^{\theta}-z_{t}^{u}\right)  ,u_{t}%
+\lambda\theta\left(  v_{t}-u_{t}\right)  \right)  d\lambda,\\
B_{z}^{\theta}\left(  t\right)   &  =%
{\displaystyle\int\nolimits_{0}^{1}}
b_{z}\left(  t,y_{t}^{u}+\lambda\left(  y_{t}^{\theta}-y_{t}^{u}\right)
,z_{t}^{u}+\lambda\left(  z_{t}^{\theta}-z_{t}^{u}\right)  ,u_{t}%
+\lambda\theta\left(  v_{t}-u_{t}\right)  \right)  d\lambda,\\
\rho_{t}^{\theta}  &  =%
{\displaystyle\int\nolimits_{0}^{1}}
\left[  b_{y}\left(  t,y_{t}^{u}+\lambda\left(  y_{t}^{\theta}-y_{t}%
^{u}\right)  ,z_{t}^{u}+\lambda\left(  z_{t}^{\theta}-z_{t}^{u}\right)
,u_{t}+\lambda\theta\left(  v_{t}-u_{t}\right)  \right)  \right. \\
&  \ \ \ \ \left.  -b_{y}\left(  t,y_{t}^{u},z_{t}^{u},u_{t}\right)  \right]
Y_{t}d\lambda\\
&  +%
{\displaystyle\int\nolimits_{0}^{1}}
\left[  b_{z}\left(  t,y_{t}^{u}+\lambda\left(  y_{t}^{\theta}-y_{t}%
^{u}\right)  ,z_{t}^{u}+\lambda\left(  z_{t}^{\theta}-z_{t}^{u}\right)
,u_{t}+\lambda\theta\left(  v_{t}-u_{t}\right)  \right)  \right. \\
&  \ \ \ \ \left.  -b_{z}\left(  t,y_{t}^{u},z_{t}^{u},u_{t}\right)  \right]
Z_{t}d\lambda\\
&  +%
{\displaystyle\int\nolimits_{0}^{1}}
\left[  b_{v}\left(  t,y_{t}^{u}+\lambda\left(  y_{t}^{\theta}-y_{t}%
^{u}\right)  ,z_{t}^{u}+\lambda\left(  z_{t}^{\theta}-z_{t}^{u}\right)
,u_{t}+\lambda\theta\left(  v_{t}-u_{t}\right)  \right)  \right. \\
&  \ \ \ \ \left.  -b_{v}\left(  t,y_{t}^{u},z_{t}^{u},u_{t}\right)  \right]
\left(  u_{t}-v_{t}\right)  d\lambda.
\end{align*}

By $\left(  5\right)  $ and $\left(  10\right)  $, it is easy to see that%
\begin{equation}
\mathbb{E}\left\vert \Phi_{t}^{\theta}\right\vert ^{2}+\mathbb{E}%
{\displaystyle\int\nolimits_{0}^{T}}
\left\vert \Psi_{t}^{\theta}\right\vert ^{2}dt<+\infty.
\end{equation}

Applying the Ito formula to $\left(  \Phi_{t}^{\theta}\right)  ^{2}$, we get%
\[
\mathbb{E}\left\vert \Phi_{t}^{\theta}\right\vert ^{2}+\mathbb{E}%
{\displaystyle\int\nolimits_{t}^{T}}
\left\vert \Psi_{s}^{\theta}\right\vert ^{2}ds\leq2\mathbb{E}%
{\displaystyle\int\nolimits_{t}^{T}}
\left\vert \Phi_{s}^{\theta}\left(  B_{y}^{\theta}\left(  s\right)  \Phi
_{s}^{\theta}+B_{z}^{\theta}\left(  s\right)  \Psi_{s}^{\theta}+\rho
_{s}^{\theta}\right)  \right\vert ds.
\]

By the Young's formula and using the fact that $B_{y}^{\theta}$ and
$B_{z}^{\theta}$ are bounded, we have for every $\varepsilon>0$%
\[
\mathbb{E}\left\vert \Phi_{t}^{\theta}\right\vert ^{2}+\mathbb{E}%
{\displaystyle\int\nolimits_{t}^{T}}
\left\vert \Psi_{s}^{\theta}\right\vert ^{2}ds\leq\left(
\genfrac{.}{.}{}{0}{1}{\varepsilon}%
+C\varepsilon\right)  \mathbb{E}%
{\displaystyle\int\nolimits_{t}^{T}}
\left\vert \Phi_{s}^{\theta}\right\vert ^{2}ds+C\varepsilon\mathbb{E}%
{\displaystyle\int\nolimits_{t}^{T}}
\left\vert \Psi_{s}^{\theta}\right\vert ^{2}ds+C\varepsilon\mathbb{E}%
{\displaystyle\int\nolimits_{t}^{T}}
\left\vert \rho_{s}^{\theta}\right\vert ^{2}ds.
\]

Choose $\varepsilon=%
\genfrac{.}{.}{}{0}{1}{2C}%
$, then we get%
\[
\mathbb{E}\left\vert \Phi_{t}^{\theta}\right\vert ^{2}+%
\genfrac{.}{.}{}{0}{1}{2}%
\mathbb{E}%
{\displaystyle\int\nolimits_{t}^{T}}
\left\vert \Psi_{s}^{\theta}\right\vert ^{2}ds\leq\left(  2C+%
\genfrac{.}{.}{}{0}{1}{2}%
\right)  \mathbb{E}%
{\displaystyle\int\nolimits_{t}^{T}}
\left\vert \Phi_{s}^{\theta}\right\vert ^{2}ds+%
\genfrac{.}{.}{}{0}{1}{2}%
\mathbb{E}%
{\displaystyle\int\nolimits_{t}^{T}}
\left\vert \rho_{s}^{\theta}\right\vert ^{2}ds.
\]

From this above inequality, we deduce two inequalities%
\begin{equation}
\mathbb{E}\left\vert \Phi_{t}^{\theta}\right\vert ^{2}\leq\left(  2C+%
\genfrac{.}{.}{}{0}{1}{2}%
\right)  \mathbb{E}%
{\displaystyle\int\nolimits_{t}^{T}}
\left\vert \Phi_{s}^{\theta}\right\vert ^{2}ds+%
\genfrac{.}{.}{}{0}{1}{2}%
\mathbb{E}%
{\displaystyle\int\nolimits_{t}^{T}}
\left\vert \rho_{s}^{\theta}\right\vert ^{2}ds.
\end{equation}%
\begin{equation}
\mathbb{E}%
{\displaystyle\int\nolimits_{t}^{T}}
\left\vert \Psi_{s}^{\theta}\right\vert ^{2}ds\leq\left(  4C+1\right)
\mathbb{E}%
{\displaystyle\int\nolimits_{t}^{T}}
\left\vert \Phi_{s}^{\theta}\right\vert ^{2}ds+\mathbb{E}%
{\displaystyle\int\nolimits_{t}^{T}}
\left\vert \rho_{s}^{\theta}\right\vert ^{2}ds.
\end{equation}

\ 

Let us prove now that $\underset{\theta\rightarrow0}{\lim}\mathbb{E}%
{\displaystyle\int\nolimits_{t}^{T}}
\left\vert \rho_{s}^{\theta}\right\vert ^{2}ds=0.$

We have
\begin{align*}
\mathbb{E}%
{\displaystyle\int\nolimits_{t}^{T}}
\left\vert \rho_{s}^{\theta}\right\vert ds  &  \leq\mathbb{E}%
{\displaystyle\int\nolimits_{t}^{T}}
{\displaystyle\int\nolimits_{0}^{1}}
\left\vert b_{y}\left(  s,y_{s}^{u}+\lambda\left(  y_{s}^{\theta}-y_{s}%
^{u}\right)  ,z_{s}^{u}+\lambda\left(  z_{s}^{\theta}-z_{s}^{u}\right)
,u_{s}+\lambda\theta\left(  v_{s}-u_{s}\right)  \right)  \right. \\
&  \ \ \ \ -\left.  b_{y}\left(  s,y_{s}^{u}+\lambda\left(  y_{s}^{\theta
}-y_{s}^{u}\right)  ,z_{s}^{u},u_{s}+\lambda\theta\left(  v_{s}-u_{s}\right)
\right)  \right\vert Y_{s}d\lambda ds\\
&  +\mathbb{E}%
{\displaystyle\int\nolimits_{t}^{T}}
{\displaystyle\int\nolimits_{0}^{1}}
\left\vert b_{y}\left(  s,y_{s}^{u}+\lambda\left(  y_{s}^{\theta}-y_{s}%
^{u}\right)  ,z_{s}^{u},u_{s}+\lambda\theta\left(  v_{s}-u_{s}\right)
\right)  \right. \\
&  \ \ \ \ -\left.  b_{y}\left(  s,y_{s}^{u},z_{s}^{u},u_{s}\right)
\right\vert Y_{s}d\lambda ds\\
&  +\mathbb{E}%
{\displaystyle\int\nolimits_{t}^{T}}
{\displaystyle\int\nolimits_{0}^{1}}
\left\vert b_{z}\left(  s,y_{s}^{u}+\lambda\left(  y_{s}^{\theta}-y_{s}%
^{u}\right)  ,z_{s}^{u}+\lambda\left(  z_{s}^{\theta}-z_{s}^{u}\right)
,u_{s}+\lambda\theta\left(  v_{s}-u_{s}\right)  \right)  \right. \\
&  \ \ \ \ -\left.  b_{z}\left(  s,y_{s}^{u}+\lambda\left(  y_{s}^{\theta
}-y_{s}^{u}\right)  ,z_{s}^{u},u_{s}+\lambda\theta\left(  v_{s}-u_{s}\right)
\right)  \right\vert Z_{s}d\lambda ds\\
&  +\mathbb{E}%
{\displaystyle\int\nolimits_{t}^{T}}
{\displaystyle\int\nolimits_{0}^{1}}
\left\vert b_{z}\left(  s,y_{s}^{u}+\lambda\left(  y_{s}^{\theta}-y_{s}%
^{u}\right)  ,z_{s}^{u},u_{s}+\lambda\theta\left(  v_{s}-u_{s}\right)
\right)  \right. \\
&  \ \ \ \ -\left.  b_{z}\left(  s,y_{s}^{u},z_{s}^{u},u_{s}\right)
\right\vert Z_{s}ds\\
&  +\mathbb{E}%
{\displaystyle\int\nolimits_{t}^{T}}
{\displaystyle\int\nolimits_{0}^{1}}
\left\vert b_{v}\left(  s,y_{s}^{u}+\lambda\left(  y_{s}^{\theta}-y_{s}%
^{u}\right)  ,z_{s}^{u}+\lambda\left(  z_{s}^{\theta}-z_{s}^{u}\right)
,u_{s}+\lambda\theta\left(  v_{s}-u_{s}\right)  \right)  \right. \\
&  \ \ \ \ -\left.  b_{v}\left(  s,y_{s}^{u}+\lambda\left(  y_{s}^{\theta
}-y_{s}^{u}\right)  ,z_{s}^{u},u_{s}+\lambda\theta\left(  v_{s}-u_{s}\right)
\right)  \right\vert \left(  u_{s}-v_{s}\right)  d\lambda ds\\
&  +\mathbb{E}%
{\displaystyle\int\nolimits_{t}^{T}}
{\displaystyle\int\nolimits_{0}^{1}}
\left\vert b_{v}\left(  s,y_{s}^{u}+\lambda\left(  y_{s}^{\theta}-y_{s}%
^{u}\right)  ,z_{s}^{u},u_{s}+\lambda\theta\left(  v_{s}-u_{s}\right)
\right)  \right. \\
&  \ \ \ \ -\left.  b_{v}\left(  s,y_{s}^{u},z_{s}^{u},u_{s}\right)
\right\vert \left(  u_{s}-v_{s}\right)  d\lambda ds.
\end{align*}

Applying the Cauchy-Schwarz inequality, then by using $\left(  4.6\right)  $
and $\left(  10\right)  $, we get%
\begin{align}
\mathbb{E}%
{\displaystyle\int\nolimits_{t}^{T}}
\left\vert \rho_{s}^{\theta}\right\vert ds  &  \leq C\left(  \mathbb{E}%
{\displaystyle\int\nolimits_{t}^{T}}
\left\vert z_{s}^{\theta}-z_{s}^{u}\right\vert ^{2}ds\right)  ^{1/2}+C\left(
\mathbb{E}%
{\displaystyle\int\nolimits_{t}^{T}}
\left\vert z_{s}^{\theta}-z_{s}^{u}\right\vert ^{2}ds\right)  ^{1/2}%
\nonumber\\
&  +C\left(  \mathbb{E}%
{\displaystyle\int\nolimits_{t}^{T}}
{\displaystyle\int\nolimits_{0}^{1}}
\left\vert b_{y}\left(  s,y_{s}^{u}+\lambda\left(  y_{s}^{\theta}-y_{s}%
^{u}\right)  ,z_{s}^{u},u_{s}+\lambda\theta\left(  v_{s}-u_{s}\right)
\right)  \right.  \right. \nonumber\\
&  \ \ \ \ -\left.  \left.  b_{y}\left(  s,y_{s}^{u},z_{s}^{u},u_{s}\right)
\right\vert ^{2}d\lambda ds\right)  ^{1/2}\\
&  +C\left(  \mathbb{E}%
{\displaystyle\int\nolimits_{t}^{T}}
{\displaystyle\int\nolimits_{0}^{1}}
\left\vert b_{z}\left(  s,y_{s}^{u}+\lambda\left(  y_{s}^{\theta}-y_{s}%
^{u}\right)  ,z_{s}^{u},u_{s}+\lambda\theta\left(  v_{s}-u_{s}\right)
\right)  \right.  \right. \nonumber\\
&  \ \ \ \ -\left.  \left.  b_{z}\left(  s,y_{s}^{u},z_{s}^{u},u_{s}\right)
\right\vert ^{2}d\lambda ds\right)  ^{1/2}\nonumber\\
&  +C\left(  \mathbb{E}%
{\displaystyle\int\nolimits_{t}^{T}}
{\displaystyle\int\nolimits_{0}^{1}}
\left\vert b_{v}\left(  s,y_{s}^{u}+\lambda\left(  y_{s}^{\theta}-y_{s}%
^{u}\right)  ,z_{s}^{u},u_{s}+\lambda\theta\left(  v_{s}-u_{s}\right)
\right)  \right.  \right. \nonumber\\
&  \ \ \ \ -\left.  \left.  b_{v}\left(  s,y_{s}^{u},z_{s}^{u},u_{s}\right)
\right\vert ^{2}d\lambda ds\right)  ^{1/2}.\nonumber
\end{align}

By $\left(  6\right)  $, the first and second terms in the right hand side of
the above inequality tends to $0$ as $\theta$ go to $0.$

On the other hand, since $b_{y},$\ $b_{z}$ and $b_{v}$ are continuous and
bounded, then from $\left(  6\right)  $ and the dominated convergence theorem,
we show that the third, fourth and fifth terms in the right hand side tends to
$0$ as $\theta$ go to $0.$

Then, we get%
\[
\underset{\theta\rightarrow0}{\lim}\mathbb{E}%
{\displaystyle\int\nolimits_{t}^{T}}
\left\vert \rho_{s}^{\theta}\right\vert ds=0.
\]

Moreover, from $\left(  15\right)  $, $\left(  5\right)  $ and the fact that
$b_{y},$\ $b_{z}$ and $b_{v}$ are bounded, we show that%
\[
\mathbb{E}\left\vert \rho_{s}^{\theta}\right\vert ds<+\infty.
\]

Using the dominated convergence theorem, we have%
\[
\underset{\theta\rightarrow0}{\lim}\mathbb{E}%
{\displaystyle\int\nolimits_{t}^{T}}
\left\vert \rho_{s}^{\theta}\right\vert ^{2}ds=0.
\]

By $\left(  13\right)  $ and Gronwall lemma, we deduce that
\[
\underset{\theta\rightarrow0}{\lim}\mathbb{E}\left\vert \Phi_{t}^{\theta
}\right\vert ^{2}=0.
\]

Finally, by $\left(  14\right)  $ we have
\[
\underset{\theta\rightarrow0}{\lim}\mathbb{E}%
{\displaystyle\int\nolimits_{0}^{T}}
\left\vert \Psi_{t}^{\theta}\right\vert ^{2}dt=0.
\]

Lemma 5 is proved.
\end{proof}

\begin{lemma}
\textit{Let }$u$\textit{\ be an optimal control minimizing the cost }%
$J$\textit{\ over }$\mathcal{U}$\textit{\ and }$\left(  y_{t}^{u},z_{t}%
^{u}\right)  $\textit{\ the solution of }$\left(  1\right)  $%
\textit{\ controlled by }$u$\textit{. Then for any }$v\in\mathcal{U}$\textit{,
we have}
\begin{align}
0  &  \leq\mathbb{E}\left[  g_{y}\left(  y_{0}^{u}\right)  Y_{0}\right]
+\mathbb{E}%
{\displaystyle\int\nolimits_{0}^{T}}
h_{y}\left(  t,y_{t}^{u},z_{t}^{u},u_{t}\right)  Y_{t}dt\\
&  +\mathbb{E}%
{\displaystyle\int\nolimits_{0}^{T}}
h_{z}\left(  t,y_{t}^{u},z_{t}^{u},u_{t}\right)  Z_{t}dt+\mathbb{E}%
{\displaystyle\int\nolimits_{0}^{T}}
h_{v}\left(  t,y_{t}^{u},z_{t}^{u},u_{t}\right)  \left(  v_{t}-u_{t}\right)
dt.\nonumber
\end{align}

\end{lemma}

\begin{proof}
We use the same notations that in lemma 5 for $\Phi_{t}^{\theta}$ and
$\Psi_{t}^{\theta}$.

Since $u$ is optimal, we have%
\begin{align*}
0  &  \leq J\left(  u^{\theta}\right)  -J\left(  u\right) \\
&  \leq\mathbb{E}\left[  g\left(  y_{0}^{\theta}\right)  -g\left(  y_{0}%
^{u}\right)  \right]  +\mathbb{E}%
{\displaystyle\int\nolimits_{0}^{T}}
\left[  h\left(  t,y_{t}^{\theta},z_{t}^{\theta},u_{t}^{\theta}\right)
-h\left(  t,y_{t}^{u},z_{t}^{u},u_{t}\right)  \right]  dt\\
&  \leq\mathbb{E}%
{\displaystyle\int\nolimits_{0}^{1}}
g_{y}\left(  y_{0}^{u}+\lambda\left(  y_{0}^{\theta}-y_{0}^{u}\right)
\right)  \left(
\genfrac{.}{.}{}{0}{y_{0}^{\theta}-y_{0}^{u}}{\theta}%
\right)  d\lambda\\
&  +\mathbb{E}%
{\displaystyle\int\nolimits_{0}^{T}}
{\displaystyle\int\nolimits_{0}^{1}}
h_{y}\left(  t,y_{t}^{u}+\lambda\left(  y_{t}^{\theta}-y_{t}^{u}\right)
,z_{t}^{u}+\lambda\left(  z_{t}^{\theta}-z_{t}^{u}\right)  ,u_{t}%
+\lambda\theta\left(  v_{t}-u_{t}\right)  \right)  \left(
\genfrac{.}{.}{}{0}{y_{t}^{\theta}-y_{t}^{u}}{\theta}%
\right)  d\lambda dt\\
&  +\mathbb{E}%
{\displaystyle\int\nolimits_{0}^{T}}
{\displaystyle\int\nolimits_{0}^{1}}
h_{z}\left(  t,y_{t}^{u}+\lambda\left(  y_{t}^{\theta}-y_{t}^{u}\right)
,z_{t}^{u}+\lambda\left(  z_{t}^{\theta}-z_{t}^{u}\right)  ,u_{t}%
+\lambda\theta\left(  v_{t}-u_{t}\right)  \right)  \left(
\genfrac{.}{.}{}{0}{z_{t}^{\theta}-z_{t}^{u}}{\theta}%
\right)  d\lambda dt\\
&  +\mathbb{E}%
{\displaystyle\int\nolimits_{0}^{T}}
{\displaystyle\int\nolimits_{0}^{1}}
h_{v}\left(  t,y_{t}^{u}+\lambda\left(  y_{t}^{\theta}-y_{t}^{u}\right)
,z_{t}^{u}+\lambda\left(  z_{t}^{\theta}-z_{t}^{u}\right)  ,u_{t}%
+\lambda\theta\left(  v_{t}-u_{t}\right)  \right)  \left(  v_{t}-u_{t}\right)
d\lambda dt.
\end{align*}

Then%
\begin{align}
0  &  \leq\mathbb{E}\left[  g_{y}\left(  y_{0}^{u}\right)  Y_{0}\right]
+\mathbb{E}%
{\displaystyle\int\nolimits_{0}^{T}}
h_{y}\left(  t,y_{t}^{u},z_{t}^{u},u_{t}\right)  Y_{t}dt\nonumber\\
&  +\mathbb{E}%
{\displaystyle\int\nolimits_{0}^{T}}
h_{z}\left(  t,y_{t}^{u},z_{t}^{u},u_{t}\right)  Z_{t}dt\\
&  +\mathbb{E}%
{\displaystyle\int\nolimits_{0}^{T}}
h_{v}\left(  t,y_{t}^{u},z_{t}^{u},u_{t}\right)  \left(  v_{t}-u_{t}\right)
dt+\delta_{t}^{\theta}.\nonumber
\end{align}
where $\delta_{t}^{\theta}$ is given by%
\begin{align*}
\delta_{t}^{\theta}  &  =\mathbb{E}\left[  \left(  g_{y}\left(  y_{0}%
^{u}+\lambda\left(  y_{0}^{\theta}-y_{0}^{u}\right)  \right)  -g_{y}\left(
y_{0}^{u}\right)  \right)  Y_{0}\right]  -\mathbb{E}%
{\displaystyle\int\nolimits_{0}^{1}}
g_{y}\left(  y_{0}^{u}+\lambda\left(  y_{0}^{\theta}-y_{0}^{u}\right)
\right)  \Phi_{0}^{\theta}d\lambda\\
&  -\mathbb{E}%
{\displaystyle\int\nolimits_{0}^{T}}
{\displaystyle\int\nolimits_{0}^{1}}
h_{y}\left(  t,y_{t}^{u}+\lambda\left(  y_{t}^{\theta}-y_{t}^{u}\right)
,z_{t}^{u}+\lambda\left(  z_{t}^{\theta}-z_{t}^{u}\right)  ,u_{t}%
+\lambda\theta\left(  v_{t}-u_{t}\right)  \right)  \Phi_{t}^{\theta}d\lambda
dt\\
&  +\mathbb{E}%
{\displaystyle\int\nolimits_{0}^{T}}
{\displaystyle\int\nolimits_{0}^{1}}
\left[  h_{y}\left(  t,y_{t}^{u}+\lambda\left(  y_{t}^{\theta}-y_{t}%
^{u}\right)  ,z_{t}^{u}+\lambda\left(  z_{t}^{\theta}-z_{t}^{u}\right)
,u_{t}+\lambda\theta\left(  v_{t}-u_{t}\right)  \right)  \right. \\
&  \ \ \ \ -\left.  h_{y}\left(  t,y_{t}^{u},z_{t}^{u},u_{t}\right)  \right]
Y_{t}d\lambda dt\\
&  -\mathbb{E}%
{\displaystyle\int\nolimits_{0}^{T}}
{\displaystyle\int\nolimits_{0}^{1}}
h_{z}\left(  t,y_{t}^{u}+\lambda\left(  y_{t}^{\theta}-y_{t}^{u}\right)
,z_{t}^{u}+\lambda\left(  z_{t}^{\theta}-z_{t}^{u}\right)  ,u_{t}%
+\lambda\theta\left(  v_{t}-u_{t}\right)  \right)  \Psi_{t}^{\theta}d\lambda
dt\\
&  +\mathbb{E}%
{\displaystyle\int\nolimits_{0}^{T}}
{\displaystyle\int\nolimits_{0}^{1}}
\left[  h_{z}\left(  t,y_{t}^{u}+\lambda\left(  y_{t}^{\theta}-y_{t}%
^{u}\right)  ,z_{t}^{u}+\lambda\left(  z_{t}^{\theta}-z_{t}^{u}\right)
,u_{t}+\lambda\theta\left(  v_{t}-u_{t}\right)  \right)  \right. \\
&  \ \ \ \ -\left.  h_{z}\left(  t,y_{t}^{u},z_{t}^{u},u_{t}\right)  \right]
Z_{t}d\lambda dt\\
&  +\mathbb{E}%
{\displaystyle\int\nolimits_{0}^{T}}
\left[  h_{v}\left(  t,y_{t}^{u}+\lambda\left(  y_{t}^{\theta}-y_{t}%
^{u}\right)  ,z_{t}^{u}+\lambda\left(  z_{t}^{\theta}-z_{t}^{u}\right)
,u_{t}+\lambda\theta\left(  v_{t}-u_{t}\right)  \right)  \right. \\
&  \ \ \ \ -\left.  h_{v}\left(  t,y_{t}^{u},z_{t}^{u},u_{t}\right)  \right]
\left(  v_{t}-u_{t}\right)  d\lambda dt.
\end{align*}

Let us show that $\underset{\theta\rightarrow0}{\lim}\delta_{t}^{\theta}=0.$

We have%
\begin{align*}
\delta_{t}^{\theta}  &  =\mathbb{E}\left[  \left(  g_{y}\left(  y_{0}%
^{u}+\lambda\left(  y_{0}^{\theta}-y_{0}^{u}\right)  \right)  -g_{y}\left(
y_{0}^{u}\right)  \right)  Y_{0}\right]  -\mathbb{E}%
{\displaystyle\int\nolimits_{0}^{1}}
g_{y}\left(  y_{0}^{u}+\lambda\left(  y_{0}^{\theta}-y_{0}^{u}\right)
\right)  \Phi_{0}^{\theta}d\lambda.\\
&  -\mathbb{E}%
{\displaystyle\int\nolimits_{0}^{T}}
{\displaystyle\int\nolimits_{0}^{1}}
h_{y}\left(  t,y_{t}^{u}+\lambda\left(  y_{t}^{\theta}-y_{t}^{u}\right)
,z_{t}^{u}+\lambda\left(  z_{t}^{\theta}-z_{t}^{u}\right)  ,u_{t}%
+\lambda\theta\left(  v_{t}-u_{t}\right)  \right)  \Phi_{t}^{\theta}d\lambda
dt\\
&  +\mathbb{E}%
{\displaystyle\int\nolimits_{0}^{T}}
{\displaystyle\int\nolimits_{0}^{1}}
\left[  h_{y}\left(  t,y_{t}^{u}+\lambda\left(  y_{t}^{\theta}-y_{t}%
^{u}\right)  ,z_{t}^{u}+\lambda\left(  z_{t}^{\theta}-z_{t}^{u}\right)
,u_{t}+\lambda\theta\left(  v_{t}-u_{t}\right)  \right)  \right. \\
&  \ \ \ \ -\left.  h_{y}\left(  t,y_{t}^{u}+\lambda\left(  y_{t}^{\theta
}-y_{t}^{u}\right)  ,z_{t}^{u},u_{t}+\lambda\theta\left(  v_{t}-u_{t}\right)
\right)  \right]  Y_{t}d\lambda dt\\
&  +\mathbb{E}%
{\displaystyle\int\nolimits_{0}^{T}}
{\displaystyle\int\nolimits_{0}^{1}}
\left[  h_{y}\left(  t,y_{t}^{u}+\lambda\left(  y_{t}^{\theta}-y_{t}%
^{u}\right)  ,z_{t}^{u},u_{t}+\lambda\theta\left(  v_{t}-u_{t}\right)
\right)  -h_{y}\left(  t,y_{t}^{u},z_{t}^{u},u_{t}\right)  \right]
Y_{t}d\lambda dt\\
&  -\mathbb{E}%
{\displaystyle\int\nolimits_{0}^{T}}
{\displaystyle\int\nolimits_{0}^{1}}
h_{z}\left(  t,y_{t}^{u}+\lambda\left(  y_{t}^{\theta}-y_{t}^{u}\right)
,z_{t}^{u}+\lambda\left(  z_{t}^{\theta}-z_{t}^{u}\right)  ,u_{t}%
+\lambda\theta\left(  v_{t}-u_{t}\right)  \right)  \Psi_{t}^{\theta}d\lambda
dt\\
&  +\mathbb{E}%
{\displaystyle\int\nolimits_{0}^{T}}
{\displaystyle\int\nolimits_{0}^{1}}
\left[  h_{z}\left(  t,y_{t}^{u}+\lambda\left(  y_{t}^{\theta}-y_{t}%
^{u}\right)  ,z_{t}^{u}+\lambda\left(  z_{t}^{\theta}-z_{t}^{u}\right)
,u_{t}+\lambda\theta\left(  v_{t}-u_{t}\right)  \right)  \right. \\
&  \ \ \ \ -\left.  h_{z}\left(  t,y_{t}^{u}+\lambda\left(  y_{t}^{\theta
}-y_{t}^{u}\right)  ,z_{t}^{u},u_{t}+\lambda\theta\left(  v_{t}-u_{t}\right)
\right)  \right]  Z_{t}d\lambda dt\\
&  +\mathbb{E}%
{\displaystyle\int\nolimits_{0}^{T}}
{\displaystyle\int\nolimits_{0}^{1}}
\left[  h_{z}\left(  t,y_{t}^{u}+\lambda\left(  y_{t}^{\theta}-y_{t}%
^{u}\right)  ,z_{t}^{u},u_{t}+\lambda\theta\left(  v_{t}-u_{t}\right)
\right)  -h_{z}\left(  t,y_{t}^{u},z_{t}^{u},u_{t}\right)  \right]
Z_{t}d\lambda dt\\
&  +\mathbb{E}%
{\displaystyle\int\nolimits_{0}^{T}}
{\displaystyle\int\nolimits_{0}^{1}}
\left[  h_{v}\left(  t,y_{t}^{u}+\lambda\left(  y_{t}^{\theta}-y_{t}%
^{u}\right)  ,z_{t}^{u}+\lambda\left(  z_{t}^{\theta}-z_{t}^{u}\right)
,u_{t}+\lambda\theta\left(  v_{t}-u_{t}\right)  \right)  \right. \\
&  \ \ \ \ -\left.  h_{v}\left(  t,y_{t}^{u}+\lambda\left(  y_{t}^{\theta
}-y_{t}^{u}\right)  ,z_{t}^{u},u_{t}+\lambda\theta\left(  v_{t}-u_{t}\right)
\right)  \right]  \left(  v_{t}-u_{t}\right)  d\lambda dt\\
&  +\mathbb{E}%
{\displaystyle\int\nolimits_{0}^{T}}
{\displaystyle\int\nolimits_{0}^{1}}
\left[  h_{v}\left(  t,y_{t}^{u}+\lambda\left(  y_{t}^{\theta}-y_{t}%
^{u}\right)  ,z_{t}^{u},u_{t}+\lambda\theta\left(  v_{t}-u_{t}\right)
\right)  \right. \\
&  \ \ \ \ -\left.  h_{v}\left(  t,y_{t}^{u},z_{t}^{u},u_{t}\right)  \right]
\left(  v_{t}-u_{t}\right)  d\lambda dt
\end{align*}

Applying the Cauchy-Schwarz inequality, then by using $\left(  10\right)
,\ \left(  4.6\right)  ,$ definition 1 and the fact that $g_{y},h_{y},h_{z}$
are bounded, we get%
\begin{align*}
\left\vert \delta_{t}^{\theta}\right\vert  &  \leq C\left(  \mathbb{E}%
\left\vert \Phi_{0}^{\theta}\right\vert ^{2}\right)  ^{1/2}+C\left(
\mathbb{E}%
{\displaystyle\int\nolimits_{0}^{T}}
\left\vert \Phi_{t}^{\theta}\right\vert ^{2}dt\right)  ^{1/2}+C\left(
\mathbb{E}%
{\displaystyle\int\nolimits_{0}^{T}}
\left\vert \Psi_{t}^{\theta}\right\vert ^{2}dt\right)  ^{1/2}\\
&  +C\left(
{\displaystyle\int\nolimits_{0}^{T}}
\mathbb{E}\left\vert z_{t}^{\theta}-z_{t}^{u}\right\vert ^{2}dt\right)
^{1/2}+C\left(  \mathbb{E}\left\vert g_{y}\left(  y_{0}^{u}+\lambda\left(
y_{0}^{\theta}-y_{0}^{u}\right)  \right)  -g_{y}\left(  y_{0}^{u}\right)
\right\vert ^{2}\right)  ^{1/2}\\
&  +C\left(  \mathbb{E}%
{\displaystyle\int\nolimits_{0}^{T}}
{\displaystyle\int\nolimits_{0}^{1}}
\left\vert h_{y}\left(  t,y_{t}^{u}+\lambda\left(  y_{t}^{\theta}-y_{t}%
^{u}\right)  ,z_{t}^{u},u_{t}+\lambda\theta\left(  v_{t}-u_{t}\right)
\right)  \right.  \right. \\
&  \ \ \ \ -\left.  \left.  h_{y}\left(  t,y_{t}^{u},z_{t}^{u},u_{t}\right)
\right\vert ^{2}d\lambda dt\right)  ^{1/2}\\
&  +C\left(  \mathbb{E}%
{\displaystyle\int\nolimits_{0}^{T}}
{\displaystyle\int\nolimits_{0}^{1}}
\left\vert h_{z}\left(  t,y_{t}^{u}+\lambda\left(  y_{t}^{\theta}-y_{t}%
^{u}\right)  ,z_{t}^{u},u_{t}+\lambda\theta\left(  v_{t}-u_{t}\right)
\right)  \right.  \right. \\
&  \ \ \ \ -\left.  \left.  h_{z}\left(  t,y_{t}^{u},z_{t}^{u},u_{t}\right)
\right\vert ^{2}d\lambda dt\right)  ^{1/2}\\
&  +C\left(  \mathbb{E}%
{\displaystyle\int\nolimits_{0}^{T}}
{\displaystyle\int\nolimits_{0}^{1}}
\left\vert h_{v}\left(  t,y_{t}^{u}+\lambda\left(  y_{t}^{\theta}-y_{t}%
^{u}\right)  ,z_{t}^{u},u_{t}+\lambda\theta\left(  v_{t}-u_{t}\right)
\right)  \right.  \right. \\
&  \ \ \ \ -\left.  \left.  h_{v}\left(  t,y_{t}^{u},z_{t}^{u},u_{t}\right)
\right\vert ^{2}d\lambda dt\right)  ^{1/2}.
\end{align*}

By $\left(  6\right)  $ and $\left(  11\right)  $, the first, second, third
and fourth terms in the right hand side of the above inequality tend to $0$ as
$\theta$ go to $0$.

On the other hand, since $g_{y},h_{y},h_{z}$ and $h_{v}$ are continous and
bounded, then by $\left(  6\right)  $ and the dominated convergence theorem,
the fifth, sixth, seventh and eighth terms in the right hand side tend to $0$
as $\theta$ to $0.$

Consequently, $\underset{\theta\rightarrow0}{\lim}\delta_{t}^{\theta}=0$ and
by letting $\theta$ go to $0$ in $\left(  17\right)  $, the proof is completed.
\end{proof}

\section{Necessary optimality conditions}

Starting from the variational inequality $\left(  16\right)  $, we can now
stated and prove our main result in thnis paper, the necessary optimality conditions.

\begin{theorem}
(Necessary optimality conditions). \textit{Let }$\left(  u,y^{u},z^{u}\right)
$\textit{\ be an optimal solution of the control problem }$\left\{  \left(
1\right)  ,\left(  2\right)  ,\left(  3\right)  \right\}  $\textit{. Then,
there exists a unique adapted process}%
\[
p^{u}\in L_{\mathcal{F}}^{2}\left(  \left[  0,T\right]  ;\mathbb{R}%
^{n}\right)  ,
\]
\textit{which is solution of the following forward stochastic differential
equation (called adjoint equation)}%
\begin{equation}
\left\{
\begin{array}
[c]{l}%
-dp_{t}^{u}=H_{y}\left(  t,y_{t}^{u},z_{t}^{u},u_{t},p_{t}^{u}\right)
dt+H_{z}\left(  t,y_{t}^{u},z_{t}^{u},u_{t},p_{t}^{u}\right)  dW_{t},\\
p_{0}^{u}=g_{y}\left(  y_{0}^{u}\right)  ,
\end{array}
\right.
\end{equation}
such that for every $v\in\mathcal{U}$%
\begin{equation}
H_{v}\left(  t,y_{t}^{u},z_{t}^{u},u_{t},p_{t}^{u}\right)  \left(  u_{t}%
-v_{t}\right)  \geq0\ ,\ as\ ,\ ae,
\end{equation}
where the Hamiltonian $H$ is defined from $\left[  0,T\right]  \times
\mathbb{R}^{n}\times\mathcal{M}_{n\times d}\left(  \mathbb{R}\right)  \times
U\times\mathbb{R}^{n}$ into $\mathbb{R}$ by%
\[
H\left(  t,y,z,v,p\right)  =pb\left(  t,y,z,v\right)  -h\left(
t,y,z,v\right)  .
\]

\end{theorem}

\begin{proof}
Since\ $p_{0}^{u}=g_{y}\left(  y_{0}^{u}\right)  $, then by the variational
inequality $\left(  16\right)  $, we have%
\begin{align}
0  &  \leq\mathbb{E}\left[  p_{0}^{u}Y_{0}\right]  +\mathbb{E}%
{\displaystyle\int\nolimits_{0}^{T}}
h_{y}\left(  t,y_{t}^{u},z_{t}^{u},u_{t}\right)  Y_{t}dt\\
&  +\mathbb{E}%
{\displaystyle\int\nolimits_{0}^{T}}
h_{z}\left(  t,y_{t}^{u},z_{t}^{u},u_{t}\right)  Z_{t}dt+\mathbb{E}%
{\displaystyle\int\nolimits_{0}^{T}}
h_{v}\left(  t,y_{t}^{u},z_{t}^{u},u_{t}\right)  \left(  v_{t}-u_{t}\right)
dt.\nonumber
\end{align}
where $\left(  Y,Z\right)  $ is the solution of $\left(  9\right)  $.

Applying the Ito formula to $p_{t}^{u}Y_{t}$, we get%
\begin{align*}
\mathbb{E}\left[  p_{0}^{u}Y_{0}\right]   &  =-\mathbb{E}%
{\displaystyle\int\nolimits_{0}^{T}}
h_{y}\left(  t,y_{t}^{u},z_{t}^{u},u_{t}\right)  Y_{t}dt-\mathbb{E}%
{\displaystyle\int\nolimits_{0}^{T}}
p_{t}^{u}b_{v}\left(  t,y_{t}^{u},z_{t}^{u},u_{t}\right)  \left(  v_{t}%
-u_{t}\right)  dt\\
&  -\mathbb{E}%
{\displaystyle\int\nolimits_{0}^{T}}
h_{z}\left(  t,y_{t}^{u},z_{t}^{u},u_{t}\right)  Z_{t}dt+\mathbb{E}\left[
S_{T}\right]  .
\end{align*}
where $S_{T}$ is given by%

\[
S_{T}=%
{\displaystyle\int\nolimits_{0}^{T}}
\left[  H_{z}\left(  t,y_{t}^{u},z_{t}^{u},u_{t},p_{t}^{u}\right)  Y-p_{t}%
^{u}Z_{t}\right]  dW_{t}.
\]

By replaces $\mathbb{E}\left[  p_{0}^{u}Y_{0}\right]  $ by it's value in
$\left(  20\right)  $, we have%
\begin{equation}
0\leq\mathbb{E}%
{\displaystyle\int\nolimits_{0}^{T}}
H_{v}\left(  t,y_{t}^{u},z_{t}^{u},u_{t}\right)  \left(  u_{t}-v_{t}\right)
dt+\mathbb{E}\left[  S_{T}\right]  .
\end{equation}

The adjoint equation $\left(  18\right)  $ is a linear forward stochastic
differential equation with bounded coefficients and bounded initial condition,
then it admits a unique adapted solution$\ p^{u}$ such that
\begin{equation}
\mathbb{E}\left[  \underset{t\in\left[  0,T\right]  }{\sup}\left\vert
p_{t}^{u}\right\vert ^{2}\right]  <+\infty.
\end{equation}

By the Cauchy-Schwarz inequality, and using $\left(  10\right)  $, $\left(
22\right)  $, the fact that $b_{z}$, $h_{z}$ are bounded and the dominated
convergence theorem, we show that $S$ is a $L^{2}$-martingale.

Hence, $\mathbb{E}\left[  S_{T}\right]  =0$ and the result follows immediately
from $\left(  21\right)  .$
\end{proof}

\end{document}